\documentclass[a4paper,11pt]{article}
\usepackage[english]{babel}
\usepackage{amsmath,amssymb,amsthm,graphicx,enumerate,mathrsfs,bbm,xcolor,latexsym, dsfont,a4wide}

\usepackage[hidelinks]{hyperref}

\newcommand{\assign}{:=}
\newcommand{\cdummy}{\cdot}
\newcommand{\mathD}{\mathrm{D}}
\newcommand{\mathd}{\mathrm{d}}

\newcommand{\tmop}[1]{\ensuremath{\operatorname{#1}}}
\newcommand{\tmtextit}[1]{{\itshape{#1}}}
\newtheorem{theorem}{Theorem}[section]

\newtheorem{definition}[theorem]{Definition}
\newtheorem{lemma}[theorem]{Lemma}
\newtheorem{remark}[theorem]{Remark}
\newtheorem{proposition}[theorem]{Proposition}

\newcommand{\CS}{\mathscr{S}}

\newcommand{\CF}{\mathscr{F}}

\newcommand{\LL}{\mathcal{L}}

\newcommand{\R}{\mathbb{R}}
\newcommand{\Z}{\mathbb{Z}}

\newcommand{\T}{\mathbb{T}}
\newcommand{\E}{\mathbb{E}}
\renewcommand{\P}{\mathbb{P}}
\newcommand{\1}{\mathds{1}}

\newcommand{\dd}{\mathrm{d}}

\begin{document}

\title{Probabilistic approach to the stochastic Burgers equation}
\date{}
\author{
  Massimiliano Gubinelli\\
  Hausdorff Center for Mathematics\\
   \& Institute for Applied Mathematics\\
   Universit{\"a}t Bonn \\
  \texttt{gubinelli@iam.uni-bonn.de}
  \and
  Nicolas Perkowski\thanks{Financial support by the DFG via Research Unit FOR 2402 is gratefully acknowledged.} \\
  Institut f\"ur Mathematik \\
  Humboldt--Universit\"at zu Berlin \\
  \texttt{perkowsk@math.hu-berlin.de}
}

\maketitle

\begin{center}
\emph{Dedicated to Michael R{\"o}ckner on the occasion of his 60th
  Birthday.}
\end{center}

\begin{abstract}
  We review the formulation of the stochastic Burgers equation as a martingale problem. One way of understanding the difficulty in making sense of the equation is to note that it is a stochastic PDE with distributional drift, so we first review how to construct finite-dimensional diffusions with distributional drift. We then present the uniqueness result for the stationary martingale problem of~\cite{GP16}, but we mainly emphasize the heuristic derivation and also we include a (very simple) extension of~\cite{GP16} to a non-stationary regime.
\end{abstract}

\section{Introduction}

In the past few years there has been a high interest and tremendous progress in so called \emph{singular stochastic PDEs} which involve very irregular noise terms (usually some form of white noise) and nonlinear operations. The presence of the irregular noise  prevents the solutions from being smooth functions and therefore the nonlinear operations appearing in the equation are a priori ill-defined and can often only be made sense of with the help of some form of renormalization. The major breakthrough in the understanding of these equations was achieved by Hairer~\cite{H12} who noted that in a one-dimensional setting rough path integration can be used to make sense of the ill-defined nonlinearities and then used this insight to solve the Kardar-Parisi-Zhang (KPZ) equation for the first time~\cite{H13KPZ} and then proceeded to develop his regularity structures~\cite{H14} which extend rough paths to higher dimensions and allow to solve a wide class of singular SPDEs. Alternative approaches are paracontrolled distributions~\cite{GIP15} and renormalization group techniques~\cite{K16}, but they all have in common that they essentially bypass probability theory and are based on pathwise arguments. A general probabilistic understanding of singular SPDEs still seems out of reach, but recently we have been able to prove uniqueness for the stationary martingale problem for the stochastic Burgers equation~\cite{GP16} (an equivalent formulation of the KPZ equation) which was introduced in 2010 by Gon\c calves and Jara~\cite{GJ10, GJ14}\footnote{The paper~\cite{GJ14} is the revised and published version of~\cite{GJ10}.}, see also~\cite{GJ13}. The aim of this short note is to give a pedagogical account of the probabilistic approach to Burgers equation. We also slightly extend the result of~\cite{GP16} to a simple non-stationary setting.

One way of understanding the difficulty in finding a probabilistic formulation to singular SPDEs is to note that they are often SPDEs with distributional drift. As an analogy, you may think of the stochastic \emph{ordinary} differential equation
  \begin{equation}\label{eq:sde}
     \dd x_t = b(x_t) \dd t + \sqrt{2} \dd w_t,
  \end{equation}
  where $x \colon \R_+ \to \R$, $w$ is a one-dimensional Brownian motion, and $b \in \CS'$, the Schwartz distributions on $\R$. It is a nontrivial problem to even make sense of~\eqref{eq:sde} because $x$ takes values in $\R$ but for a point $z \in \R$ the value $b(z)$ is not defined. One possible solution goes as follows: If there exists a nice antiderivative $B'=b$, then we consider the measure $\mu(\dd x) = (\exp( B(x)) /Z)\dd x$ on $\R$, where $Z>0$ is a normalization constant. If $b$ itself is a continuous function, then this measure is invariant with respect to the dynamics of~\eqref{eq:sde} and the solution $x$ corresponds to the Dirichlet form $\mathcal{E}(f,g) = \int \partial_x f(x) \partial_x g(x) \mu (\dd x)$. But $\mathcal{E}$ makes also sense for distributional $b$, so in that case the solution to~\eqref{eq:sde} can be simply defined as the Markov process corresponding to~$\mathcal{E}$. Among others, the works~\cite{Mathieu1994, Mathieu1995} are based on this approach.
  
  An alternative viewpoint is to consider the martingale problem associated to~\eqref{eq:sde}. The infinitesimal generator of $x$ should be $\LL = b \partial_x + \partial_{xx}^2$ and we would like to find a sufficiently rich space of test functions $\varphi$ for which
  \[
     M^\varphi_t = \varphi(x_t) - \varphi(x_0) - \int_0^t \LL \varphi(x_s) \dd s
  \]
  is a continuous martingale. But now our problem is that the term $b \partial_x \varphi$ appearing in $\LL \varphi$ is the product of a distribution and a smooth function. Multiplication with a non-vanishing smooth function does not increase regularity (think of multiplying by $1$ which is perfectly smooth), and therefore $\LL \varphi$ will only be a distribution and not a function. But then the expression $ \int_0^t \LL \varphi(x_s) \dd s$ still does not make any sense! The solution is to take non-smooth functions $\varphi$ as input into the generator because multiplication with a non-smooth function can increase regularity (think of multiplying an irregular function $f > 0$ with $1/f$): If we can solve the equation $\LL \varphi = f$ for a given continuous function $f$, then we can reformulate the martingale problem by requiring that
  \[
     M^\varphi_t = \varphi(x_t) - \varphi(x_0) - \int_0^t f(x_s) \dd s
  \]
  is a continuous martingale. If we are able to solve the equation $\LL \varphi = f$ for all continuous bounded functions $f$, then we explicitly obtain the domain of the generator $\LL$ as $\mathrm{dom}(\LL) = \{ \varphi: \LL \varphi = f \text{ for some } f \in C_b\}$. In that case the distribution of $x$ is uniquely determined through the martingale problem. Of course $\LL \varphi = f$  is still not easy to solve because $\LL \varphi$ contains the term $b \partial_x \varphi$ which is a product between the distribution $b$ and the non-smooth function $\partial_x \varphi$ and therefore not always well defined. In the one-dimensional time-homogeneous framework that we consider here it is however possible to apply a transformation that removes this ill-behaved term and maps the equation $\LL \varphi = f$ to a well posed PDE provided that $b$ is the derivative of a continuous function. This was carried out in the very nice papers~\cite{FRW03, FRW04}, where it was also observed that the solution is always a Dirichlet process (sum of a martingale plus a zero quadratic variation term), and that if $(b_n)$ is a sequence of smooth functions converging in $\CS'$ to $b$, then
  \begin{equation}\label{eq:martingale-approx}
     x_t = x_0 + \lim_{n \to \infty} \int_0^t b_n(x_s) \dd s + \sqrt{2} w_t.
  \end{equation}
  So while $b(x_s)$ at a fixed time $s$ does not make any sense, there is a time-decorrelation effect happening that allows to define the integral $\int_0^t b(x_s) \dd s$ using the above limit procedure.
  
The transformation of the PDE breaks down as soon as we look at multi-dimensional or time-inhomogeneous diffusions. But the philosophy of solving $\LL \varphi = f$ to identify the domain of the generator carries over. It is then necessary to deal with PDEs with distributional coefficients, and using paraproducts, rough paths, and paracontrolled distributions respectively, the construction of $x$ could recently be extended to the multi-dimensional setting for comparably regular $b$ (H\"older regularity $> -1/2$)~\cite{FIR17}, a more irregular one-dimensional setting~\cite{DD16} (H\"older regularity $>-2/3$), and finally the irregular multi-dimensional setting~\cite{CC15} (H\"older regularity $>-2/3$). In all these works the time-homogeneity no longer plays a role. Let us also mention that all these works are concerned with probabilistically weak solutions to the equation. If $b$ is a non-smooth function rather than a distribution, then it is possible to find a unique probabilistically strong solution. This goes back to~\cite{Z74, V80} and very satisfactory results are due to Krylov and R\"ockner~\cite{KR05}. For a pathwise approach that extends to non-semimartingale driving noises such as fractional Brownian motion see also~\cite{CG15}. A good recent survey on such ``regularization by noise'' phenomena is~\cite{F10}.

Having gained some understanding of the finite-dimensional case, we can now have a look at the stochastic Burgers equation on the one-dimensional torus $\T = \R / \Z$, $u \colon \R_+ \times \T \to \R$,
\begin{equation}\label{eq:burgers}
   \dd u_t = \Delta u_t \dd t + \partial_x u^2_t \dd t + \sqrt{2} \dd \partial_x W_t,
\end{equation}
where $\partial_t W$ is a space-time white noise, that is the centered Gaussian process with covariance $\E[\partial_t W(t,x) \partial_t W(s,y)] = \delta(t-s) \delta(x-y)$. We would like to understand~\eqref{eq:burgers} as an infinite-dimensional analogue of~\eqref{eq:sde} for a particular $b$. It is known that the dynamics of $u$ are invariant under the measure $\mu$, the law of the space white noise on $\T$, which means that for fixed times $u_t$ is a distribution with regularity $C^{-1/2-\varepsilon}$ (i.e. the distributional derivative of a $(1/2-\varepsilon)$ H\"older continuous function) and not a function. But then the square $u^2_t$ does not make any sense analytically because in general it is impossible to multiply distributions in $C^{-1/2-\varepsilon}$. Let us however simply ignore this problem and try to use the probabilistic structure of $u$ to make sense of $u^2$. For simplicity we start $u$ in the stationary measure $\mu$. Then $u_t^2$ is simply the square of the white noise, and this is something that we can explicitly compute. More precisely, let $\rho \in C^\infty_c(\R)$ be a nonnegative function with $\int_\R \rho(x) \dd x = 1$ and define the convolution
\[
  (\rho^N \ast u)(x) := \sum_{k \in \Z} u(N \rho(N(x - k -\cdot)) .
\]
Then $\rho^N \ast u$ is a smooth periodic function that converges to $u$, and in particular $(\rho^N \ast u)^2$ is well defined. Let us test against some $\varphi \in C^\infty$ and consider $(\rho^N \ast u)^2(\varphi) = \int_\T (\rho^N \ast u)^2(x) \varphi(x) \dd x$. The first observation we then make is that $\E[\int_\T (\rho^N \ast u)^2(x) \varphi(x) \dd x] = c_N \int_\T \varphi(x) \dd x$ with a diverging constant $c_N \to \infty$. So to have any hope to obtain a well-defined limit, we should rather consider the renormalized $((\rho^N \ast u)^2 - c_N)(\varphi)$. Of course, the constant $c_N$ vanishes under differentiation so we will not see it in the equation that features the term $\partial_x u^2$ and not $u^2$. Now that we subtracted the mean $\E[(\rho^N \ast u)^2(\varphi)]$, we would like to show that the variance of the centered random variable stays uniformly bounded in $N$. This is however not the case, and in particular $((\rho^N \ast u)^2 - c_N)(\varphi)$ does not converge in $L^2(\mu)$. But $L^2(\mu)$ is a Gaussian Hilbert space and $((\rho^N \ast u)^2 - c_N)$ lives for all $N$ in the second (homogeneous) chaos. Since for sequences in a fixed Gaussian chaos convergence in probability is equivalent to convergence in $L^2$ (see~\cite{J97}, Theorem~3.50), $((\rho^N \ast u)^2 - c_N)(\varphi)$ does not converge in any reasonable sense. Therefore, the square $u^2(\varphi)$ (or rather the renormalized square $u^{\diamond 2} (\varphi) = (u^2 - \infty) (\varphi)$) cannot be declared as a random variable! However, and as far as we are aware it was Assing~\cite{A02} who first observed this, the renormalized square does make sense as a distribution on $L^2(\mu)$. More precisely, if $F(u)$ is a ``nice'' random variable (think for example of $F(u) = f(u(\varphi_1), \dots, u(\varphi_n))$ for $f \in C^\infty_c(\R)$ and $\varphi_k \in C^\infty(\T)$, $k = 1, \dots, n$), then $\E[((\rho^N \ast u)^2 - c_N)(\varphi) F(u)]$ converges to a limit that we denote with $\E[u^{\diamond 2}(\varphi) F(u)]$ and that does not depend on the mollifier $\rho$. This means that we cannot evaluate $u \mapsto u^{\diamond 2}(\varphi)$ pointwise, but it makes sense when tested against smooth random variables by evaluating the integral $\int_{C^{-1/2-\varepsilon}} u^{\diamond 2}(\varphi) F(u) \mu(\dd u)$. Compare that with a distribution $T \in \CS'(\R)$ on $\R$ for which the pointwise evaluation $T(x)$ makes no sense, but we can test $T$ against smooth test functions $\psi$ by evaluating the integral $\int_\R T(x) \psi(x) \dd x$. Of course  in finite dimensions the integration is canonically performed with respect to  the Lebesgue measure, while in infinite dimensions we have to pick a reference measure which here is the invariant measure $\mu$.

So $u^{\diamond 2}(\varphi)$ is a distribution on the infinite-dimensional space $C^{-1/2-\varepsilon}$ and therefore the stochastic Burgers equation~\eqref{eq:burgers} is an infinite-dimensional analog to~\eqref{eq:sde}. We would like to use the same tools as in the finite dimensional case, and in particular we would like to make sense of the martingale problem for $u$. However, while in finite dimensions it is possible to solve the equation $\LL \varphi = f$ under rather general conditions, now this would be an infinite-dimensional PDE with a distributional coefficient. There exists a theory of infinite-dimensional PDEs, see for example~\cite{DPZ02, BKRS15}, but at the moment it seems an open problem how to solve such PDEs with distributional coefficients. On the other side, if we simply plug in a smooth function $F$ into the formal generator $\LL$, then $\LL F$ is only a distribution and not a function and therefore no smooth $F$ can be in the domain of the Burgers generator. Assing~\cite{A02} avoided this problem by considering a ``generalized martingale problem'', where the drift term $\int_0^t \partial_x u_s^2 \dd s = \int_0^t \partial_x u_s^{\diamond 2} \dd s$ never really appears. However, it is still not known whether solutions to the generalized martingale problem exist and are unique. This was the state until 2010, when Gon\c calves and Jara~\cite{GJ10} introduced a new notion of martingale solution to the stochastic Burgers equation by stepping away from the infinitesimal approach based on the generator and by making use of the same time-decorrelation effect that allowed us to construct $\int_0^t b(x_s) \dd s$ in~\eqref{eq:martingale-approx} despite the fact that $b(x_s)$ itself makes no sense. In the next chapter we discuss their approach and how to prove existence and uniqueness of solutions.
  
\section{Energy solutions to the stochastic Burgers equation}

We would now like to formulate a notion of martingale solution to the stochastic Burgers equation
\[
   \dd u_t = \Delta u_t \dd t + \partial_x u^2_t \dd t + \sqrt{2} \dd \partial_x W_t.
\]
As discussed above the infinitesimal picture is not so clear because the drift is a distribution in an infinite-dimensional space. So the idea of Gon{\c c}alves and Jara~{\cite{GJ14}} is to rather translate the formulation~\eqref{eq:martingale-approx} to our context. Let us call $u$ a \tmtextit{martingale solution} if $u \in C
(\mathbb{R}_+, \CS')$, where $\CS' = \CS'(\T)$ are the (Schwartz) distributions on $\T$, and if there exists $\rho \in C^\infty_c(\R)$ with
$\int_{\mathbb{R}} \rho (x) \mathd x = 1$ such that with $\rho^N = N \rho (N
\cdummy)$ the limit
\[ \int_0^{\cdummy} \partial_x u_s^2 \mathd s (\varphi) \assign \lim_{N
   \rightarrow \infty} \int_0^{\cdummy} (u_s \ast \rho^N)^2 (- \partial_x
   \varphi) \mathd s \]
exists for all $\varphi \in C^\infty(\T)$, uniformly on compacts in probability. Note that because of the derivative $\partial_x$ here it is not necessary to renormalize the square by subtracting a large constant.
Moreover, we require that for all $\varphi \in C^\infty(\T)$ the process
\begin{equation}\label{eq:martingale}
   M_t (\varphi) \assign u_t (\varphi) - u_0 (\varphi) - \int_0^t u_s (\Delta
   \varphi) \mathd s - \int_0^t \partial_x u_s^2 \mathd s (\varphi), \qquad t
   \geqslant 0,
\end{equation}
is a continuous martingale with quadratic variation $\langle M (\varphi)
\rangle_t = 2 \| \partial_x \varphi \|_{L^2}^2 t$. It can be easily shown that martingale solutions exist, for example by considering Galerkin approximations to the stochastic Burgers equation.

On the other side, martingale solutions are very weak and give us absolutely no control on
the nonlinear part of the drift, so it does not seem possible to show that they are unique. To overcome this we should add further conditions
to the definition, and Gon{\c c}alves and Jara assumed in~{\cite{GJ14}} additionally that $u$ satisfies an \tmtextit{energy
estimate}, that is
\[ \mathbb{E} \left[ \left( \int_s^t \{ (u_r \ast \rho^N)^2 (- \partial_x
   \varphi) - (u_r \ast \rho^M)^2 (- \partial_x \varphi) \} \mathd r \right)^2
   \right] \lesssim \frac{(t - s)}{M \wedge N} \| \partial_x \varphi
   \|_{L^2}^2 . \]
The precise form of the estimate is not important, we will only need that it
implies $\int_0^{\cdummy} \partial_x u_s^2 \mathd s \in C (\mathbb{R}_+,
\CS')$ and that for every test function $\varphi$ the process
$\int_0^{\cdummy} \partial_x u_s^2 \mathd s (\varphi)$ has zero quadratic
variation. Consequently, $u (\varphi)$ is a Dirichlet process (sum of a continuous local martingale and a zero quadratic variation process) and admits an
It{\^o} formula~{\cite{RV}}. Any martingale solution that satisfies the
energy estimate is called an \tmtextit{energy solution}.

This is still a very weak notion of solution, and there does not seem to be
any way to directly compare two different energy solutions and to show that
they have the same law. To see why the additional structure coming from the
energy estimate is nonetheless very useful, let us argue formally for a
moment. Let $u$ be a solution to Burgers equation $\partial_t u = \Delta u +
\partial_x u^2 +  \partial_t \partial_x W$. Then $u = \partial_x h$ for the solution $h$
to the KPZ equation
\[ \partial_t h = \Delta h + (\partial_x h)^{\diamond 2} + \partial_t W = \Delta h + (\partial_x h)^{2} - \infty + \partial_t W . \]
But the KPZ equation can be linearized through the Cole-Hopf transform~\cite{BG97}: We
formally have $h = \log \phi$ for the solution $w$ to the linear stochastic heat
equation
\begin{equation}\label{eq:SHE}
   \partial_t \phi = \Delta \phi + \phi \partial_t W,
\end{equation}
which is a linear It{\^o} SPDE that can be solved using classical
theory~{\cite{W86,DPZ14,PR07}}. So if we can show that every energy
solution $u$ satisfies $u = \partial_x \log \phi$ where $\phi$ solves~\eqref{eq:SHE}, then
$u$ is unique. And the key difference between energy solutions and martingale solutions is that for energy solutions we have an It\^o formula that allows us to perform a change of variables that maps $u$ to a candidate solution to the stochastic heat equation.

So let $u$ be an energy solution, let $\rho \in \CS (\mathbb{R})$ be an even function with Fourier transform $\hat{\rho} = \int_\R e^{2 \pi i x \cdot} \rho(x) \dd x \in C^\infty_c (\mathbb{R})$ and such that $\hat{\rho} \equiv 1$ on a neighborhood of $0$, and define
\begin{equation}\label{eq:uL-def-periodic}
   u^L_t \assign \CF^{-1}_\T(\hat{\rho}(L^{-1}\cdot) \CF_\T u_t ) = \rho^L \ast u_t, \qquad t \geqslant 0
\end{equation}
where $\CF_\T u(k) \assign \int_\T e^{2 \pi i k x} u(x) \dd x$ respectively $\CF^{-1}_\T \psi(x) \assign \sum_{k \in \Z} e^{2\pi i k x} \psi(k)$ denote the Fourier transform (respectively inverse Fourier transform) on the torus. We then integrate $u^L$ by setting
\[
   h^L_t \assign \CF^{-1}_\T (\CF_\T{\Theta} \CF_\T(u^L_t)) = \Theta \ast u^L_t = (\Theta \ast {\rho}^L) \ast u_t =: \Theta^L \ast u_t, \qquad t \geqslant 0,
\]
where $\CF_\T{\Theta}(k) = \1_{k \neq 0} (2\pi i k)^{-1}$ and therefore $\partial_x (\Theta \ast u) = u - \int_\T u(x) \dd x$. Since $\int_\T u(x) \dd x$ is conserved by the stochastic Burgers equation, it suffices to prove uniqueness of $\partial_x (\Theta \ast u)$. Writing $h^L_t(x) = u_t(\Theta^L_x)$ for $\Theta^L_x(y) \assign (\Theta \ast {\rho}^L)(x - y)$, we get from~\eqref{eq:martingale}
\begin{align*}
   \dd h^L_t (x) 
   & = \Delta_x h^L_t(x) + u^2_t (\partial_x \Theta^L_x) \dd t + \sqrt{2} \dd W_t (\partial_x \Theta^L_x)
\end{align*}
for a white noise $W$, and $\dd \langle h^L(x) \rangle_t = 2 \| \partial_x \Theta^L_x \|_{L^2(\T)}^2 \dd t$. It is not hard to see that $\partial_x \Theta^L_x = {\rho}^L_x - 1$ for ${\rho}^L_x(y) = {\rho}^L(x - y)$. So setting $\phi^L_t (x) \assign e^{ h^L_t (x)}$ we have by the It\^o formula for Dirichlet processes~\cite{RV} and a short computation
\begin{align*}
  \dd \phi^L_t (x) & = \Delta_x \phi^L_t (x) \dd t + \sqrt{2}  \phi^L_t (x) \dd W_t ({\rho}^L_x) +  \dd R^L_t(x) +  K \phi^L_t(x) \dd t + \phi^L_t(x) \dd Q^L_t \\
  &\quad - \sqrt{2}  \phi^L_t(x) \dd W_t(1) - 2 \phi^L_t(x) \dd t,
\end{align*}
where
\begin{equation}\label{eq:RL}
   R^L_t (x) \assign \int_0^t \phi^L_s (x) \Big\{u^2_s (\partial_x \Theta^L_x) - \Big((u_s^L (x))^2 - \int_\T (u_s^L (y))^2 \dd y\Big) - K \Big\} \dd s
\end{equation}
for a suitable constant $K$ and
\begin{equation}\label{eq:QL}
  Q^L_t \assign \int_0^t \Big\{ - \int_\T((u^L_s)^2(y) - \| {\rho}^L \|_{L^2(\R)}^2) \dd y + 1 \Big\} \dd s.
\end{equation}

\begin{proposition}\label{prop:uniqueness}
   Let $u \in C (\mathbb{R}_+, \CS')$ be an
  energy solution to the stochastic Burgers equation such that
    \begin{equation}\label{eq:exp-control} \sup_{x \in \T, t \in [0,T]} \mathbb{E} [|e^{u_t(\Theta_x)}|^2] < \infty. \end{equation}
   If for $R^L$ and $Q^L$ defined in~\eqref{eq:RL} and~\eqref{eq:QL} respectively and every test function $\varphi$ the process $R^L(\varphi)$ converges to zero uniformly on compacts in probability and $Q^L$ converge to a zero quadratic variation process $Q$, then $u$ is unique in law and given by $u_t = \partial_x \log \phi_t + \int_\T u_0(y) \dd y$, where $\phi$ is the unique solution to the stochastic heat equation
  \[
     \dd \phi_t = \Delta \phi_t \dd t + \phi_t \dd W_t, \qquad \phi_0(x) = e^{u_0(\Theta_x)}.
  \]
\end{proposition}

\begin{proof}
  This is Theorem~2.13 in~{\cite{GP16}}. Actually we need a slightly stronger convergence of $R^L(\varphi)$ and $Q^L$ than locally uniform, but to simplify the presentation we ignore this here.
\end{proof}

\begin{remark}
  The strategy of mapping the energy solution to the linear stochastic heat equation is
  essentially due to Funaki and Quastel~\cite{FQ15}, who used it in a different context to study the invariant measure of the KPZ equation. In their approach
  similar correction terms as $R^L$ and $Q^L$ appeared, and the fact
  that they were able to deal with them gave us the courage to proceed with
  the rather long computations that control $R^L$ and $Q^L$ in our setting.
\end{remark}

So to obtain the uniqueness of energy solutions we need to verify the
assumptions of Proposition~\ref{prop:uniqueness}. Unfortunately we are not able to do
this because while the energy condition gives us good control of the Burgers
nonlinearity $\int_0^{\cdummy} \partial_x u_s^2 \mathd s$, it does not allow
us to bound general additive functionals $\int_0^{\cdummy} F (u_s) \mathd s$
such as the ones appearing in the definition of $R^L$.

To understand which condition to add in order to control such additive
functionals, let us recall how this can be (formally) achieved for a Markov
process $X$ with values in a Polish space. Assume that $X$ is stationary with
initial distribution $\mu$, denote its infinitesimal generator with
$\mathcal{L}$, and let $\mathcal{L}^{\ast}$ be the adjoint of $\mathcal{L}$ in
$L^2 (\mu)$. Then $\mathcal{L}$ can be decomposed into a symmetric part
$\mathcal{L}_S = (\mathcal{L}+\mathcal{L}^{\ast}) / 2$ and an antisymmetric
part $\mathcal{L}_A = (\mathcal{L}-\mathcal{L}^{\ast}) / 2$. Moreover, if we
reverse time and set $\hat{X}_t = X_{T - t}$ for some fixed $T > 0$, then
$(\hat{X}_t)_{t \in [0, T]}$ is a Markov process in its natural filtration,
the backward filtration of $X$, and it has the generator $\mathcal{L}^{\ast}$.
See Appendix~1 of~{\cite{KL13}} for the case where $X$ takes values in
a countable space. So by Dynkin's formula for $F \in \tmop{dom}
(\mathcal{L})\cap\mathrm{dom}(\mathcal{L}^\ast)$ the process
\[ M^F_t = F (X_t) - F (X_0) - \int_0^t \mathcal{L}F (X_s) \mathd s, \qquad t
   \geqslant 0, \]
is a martingale, and
\[ \hat{M}^F_t = F (\hat{X}_t) - F (\hat{X}_0) - \int_0^t \mathcal{L}^{\ast} F
   (\hat{X}_s) \mathd s, \qquad t \in [0, T], \]
is a martingale in the backward filtration. We add these two formulas and
obtain the following decomposition which is reminiscent of the \tmtextit{Lyons-Zheng decomposition} of a Dirichlet process into a sum of forward martingale and a backward martingale:
\[ M^F_t + (\hat{M}^F_T - \hat{M}^F_{T - t}) = - 2 \int_0^t \mathcal{L}_S F
   (X_s) \mathd s, \]
so any additive functional of the form $\int_0^t
\mathcal{L}_S F (X_s) \mathd s$ is the sum of two martingales, one in the
forward filtration, the other one in the backward filtration. The predictable
quadratic variation of the martingale $M^F$ is
\[ \langle M^F \rangle_t = \int_0^t (\mathcal{L}F^2 (X_s) - 2 F (X_s)
   \mathcal{L}F (X_s)) \mathd s, \]
and
\[ \langle \hat{M}^F \rangle_t = \int_0^t (\mathcal{L}^{\ast} F^2 (\hat{X}_s)
   - 2 F (\hat{X}_s)\mathcal{L}^{\ast} F (\hat{X}_s)) \mathd s. \]
So the Burkholder-Davis-Gundy inequality gives for all $p \geqslant 2$
\begin{align*}
  \mathbb{E} \left[ \sup_{t \in [0, T]} \left| \int_0^t \mathcal{L}_S F (X_s)
  \mathd s \right|^p \right] & \lesssim T^{p / 2 - 1} \int_0^T \mathbb{E}
  [(\mathcal{L}F^2 (X_s) - 2 F (X_s) \mathcal{L}F (X_s))^{p / 2}] \mathd s\\
  &\qquad + T^{p / 2 - 1} \int_0^T \mathbb{E} [(\mathcal{L}^{\ast} F^2 (\hat{X}_s) -
  2 F (\hat{X}_s)\mathcal{L}^{\ast} F (\hat{X}_s))^{p / 2}] \mathd s\\
  & = T^{p / 2} \mathbb{E} [(\mathcal{L}F^2 (X_0) - 2 F (X_0) \mathcal{L}F
  (X_0))^{p / 2}] \\
  &\qquad + T^{p / 2}\mathbb{E} [(\mathcal{L}^{\ast} F^2 (X_0) - 2 F (X_0)
  \mathcal{L}^{\ast} F (X_0))^{p / 2}] ,
\end{align*}
where in the last step we used the stationarity of $X$. Moreover, if
$\mathcal{L}_A$ satisfies Leibniz rule for a first order differential
operator, then $\mathcal{L}_AF^2  - 2 F  \mathcal{L}_A F = 0$ and therefore
\[ \mathbb{E} \left[ \sup_{t \in [0, T]} \left| \int_0^t \mathcal{L}_S F
   (X_s) \mathd s \right|^p \right] \lesssim T^{p / 2} \mathbb{E}
   [(\mathcal{L}_S F^2 (X_0) - 2 F (X_0) \mathcal{L}_S F (X_0))^{p / 2}] . \]
That is, we can bound additive functionals of a stationary Markov process $X$
in terms of the symmetric part of the generator $\mathcal{L}_S$ and the
invariant measure $\mu$. We call this inequality the \tmtextit{martingale trick}.

For the stochastic Burgers equation this promises to be very powerful, because
formally an invariant measure is given by the law of the spatial white noise and the symmetric part of the generator is simply
the generator $\mathcal{L}_S$ of the Ornstein-Uhlenbeck process $\dd X_t = \Delta X_t \dd t + \dd \partial_x  W_t$, that is the linearized Burgers equation. This also suggests that
$\mathcal{L}_A$ is a first order differential operator because it corresponds
to the drift $\partial_x u^2$. We thus need a notion of solution to the
stochastic Burgers equation which allows us to make the heuristic
argumentation above rigorous. This definition was given by Gubinelli and Jara
in~{\cite{GJ13}}:

\begin{definition}
  Let $u$ be an energy solution to the stochastic Burgers equation. Then $u$
  is called a \tmtextit{forward-backward (FB) solution} if additionally the
  law of $u_t$ is that of the white noise for all $t \geqslant 0$, and for all
  $T > 0$ the time-reversed process $\hat{u}_t = u_{T - t}$, $t \in [0, T]$,
  is an energy solution to
  \[ \dd \hat{u}_t = \Delta \hat{u}_t \dd t - \partial_x \hat{u}^2_t \dd t + \sqrt{2} \dd \partial_x
     \hat{W}_t, \]
  where $\partial_t \hat{W}$ is a space-time white noise in the backward filtration.
\end{definition}

\begin{remark}
  {\cite{GJ13}} do not define FB-solutions, but they also call their
  solutions energy solutions. For pedagogical reasons we prefer here to
  introduce a new terminology for this third notion of solution. Also, the
  definition in~{\cite{GJ13}} is formulated slightly differently than
  above, but it is equivalent to our definition.
\end{remark}

Of course, we should first verify whether FB-solutions exist before we proceed to discuss their uniqueness. But existence is very easy and can be shown by Galerkin approximation. Also, it is known for a wide class of interacting particle systems that they are relatively compact under rescaling and all limit points are FB-solutions to Burgers equation~\cite{GJ14, GJS15, DGP16}. A general methodology that allows to prove similar results for many particle systems was developed in~\cite{GJ14}. So all that is missing in the convergence proof is the uniqueness of FB-solutions.

\begin{lemma}
  If $u$ is a FB-solution, then the martingale trick works:
  \[ \mathbb{E} \left[ \sup_{t \in [0, T]} \left| \int_0^t \mathcal{L}_S F
     (u_s) \mathd s \right|^p \right] \lesssim T^{p / 2} \mathbb{E}
     [\mathcal{E} (F (u_0))^{p / 2}], \]
  where $\mathcal{E} (F (u)) = 2 \int_{\mathbb{R}} | \partial_x \mathD_x F (u) |^2 \mathd x$ for the Malliavin derivative $\mathD$ which is defined in terms of the law of the white noise.
\end{lemma}

\begin{proof}
  This is Lemma~2 in~\cite{GJ13} or Proposition~3.2 in~{\cite{GP16}}.
\end{proof}

Given a FB-solution $u$ we can thus control any additive functional of the form $\int_0^{\cdummy}
\mathcal{L}_S F (u_s) \mathd s$. To bound $\int_0^{\cdummy} G (u_s) \mathd s$
for a given $G$ we therefore have to solve the Poisson equation $\mathcal{L}_S
F = G$. This is an infinite-dimensional partial differential equation and a
priori difficult to understand. However, we have to only solve it in $L^2
(\mu)$ which has a lot of structure as a Gaussian Hilbert space. In particular
we have the chaos decomposition
\[ F = \sum_{n = 0}^{\infty} W_n (f_n) \]
for all $F \in L^2 (\mu)$, where $f_n \in L^2 (\mathbb{R}^n)$ and $W_n (f_n)$
is the $n$-th order Wiener-It{\^o} integral over $f_n$. And the
action of the Ornstein-Uhlenbeck generator  on $W_n (f_n)$ has
a very simple form:
\[ \mathcal{L}_S W_n (f_n) = W_n (\Delta f_n), \]
where $\Delta = \partial_{x_1}^2 + \ldots + \partial_{x_n}^2$ is the Laplace
operator on $\mathbb{R}^n$, see Lemma~3.7 in~{\cite{GP16}}. This
reduces the equation $\mathcal{L}_S F = G$ to the infinite system of uncoupled
equations
\[ \Delta_n f_n = g_n, \qquad n \in \mathbb{N}_0, \]
where $G = \sum_n W_n (g_n)$.

To test the tools we have developed so far, let us apply them to the Burgers
nonlinearity:

\begin{lemma}
  Let $u$ be a FB-solution. There exists a unique process $\int_0^{\cdummy}
  u_s^{\diamond 2} \mathd s \in C (\mathbb{R}_+, \CS')$ such that
  $\partial_x \int_0^{\cdummy} u_s^{\diamond 2} \mathd s = \int_0^{\cdummy}
  \partial_x u_s^2 \mathd s$ and such that for all $T, p > 0$, $\alpha \in (0,
  3 / 4)$, and $\chi, \varphi \in \mathcal{S}$ with $\int_{\mathbb{R}} \chi
  (x) \mathd x = 1$ we have with $\chi^N = N \chi (N \cdummy)$
  \[ \lim_{N \rightarrow \infty} \mathbb{E} \left[ \left\| \int_0^{\cdummy}
     u_s^{\diamond 2} \mathd s (\varphi) - \int_0^{\cdummy} ((u_s \ast
     \chi^N)^2 - \| \chi^N \|_{L^2}^2) (\varphi) \mathd s
     \right\|^p_{C^{\alpha} ([0, T], \mathbb{R})} \right] = 0. \]
\end{lemma}

\begin{proof}
  This is a combination of Proposition~3.15 and Corollary~3.17
  in~{\cite{GP16}}.
\end{proof}

The H{\"o}lder regularity of $\int_0^{\cdummy} \partial_x u_s^2 \mathd s (\varphi)$
for $\varphi \in \mathcal{S}$ is indeed only $3 / 4 - \varepsilon$ and not
better than that. This means that the process $u (\varphi)$ is not a
semimartingale. In particular, if we assume for
the moment that the stochastic Burgers equation defines a Markov process on a
suitable Banach space of distributions, then the map $u \mapsto u (\varphi)$
for $\varphi \in \CS$ is not in the domain of its generator. In fact it is an open problem to find any nontrivial
function in the domain of the Burgers generator.

Next, we would like to use the martingale trick to prove the convergence of
$R^L$ and $Q^L$ that we need in Proposition~\ref{prop:uniqueness}. However, while
for the Burgers nonlinearity the additive functional was of the form $\int_0^{\cdummy} G (u_s) \mathd s$ with a $G$ that has only
components in the second Gaussian chaos, the situation is not so simple for
$R^L$ because the factor $e^{u (\Theta^L_x)}$ has an infinite chaos
decomposition. While in principle it is still possible to write down the chaos
decomposition of $R^L$ exlicitly, we prefer to follow another route. In fact
there is a general tool for Markov processes which allows to control additive
functionals without explicitly solving the Poisson equation, the so called
\emph{Kipnis-Varadhan estimate}. It is based on duality and reads
\[ \mathbb{E} \left[ \sup_{t \in [0, T]} \left| \int_0^t F (u_s) \mathd s
   \right|^2 \right] \lesssim T \| F \|_{- 1}^2 \]
for
\[ \| F \|_{- 1}^2 = \sup_G \{ 2\mathbb{E} [F (u_0) G (u_0)] +\mathbb{E} [F
   (u_0) \mathcal{L}_S F (u_0)] \} . \]
Corollary~3.5 in~{\cite{GP16}} proves that this inequality holds also for FB-solutions, despite the fact
that we do not know yet whether they are Markov processes. In fact FB-solutions give us enough flexibility
to implement the classical proof from the Markovian setting, as presented for
example in~{\cite{KLO12}}. Based on the Kipnis-Varadhan inequality, Gaussian
integration by parts, and a lenghty computation we are able to show the
following result:

\begin{theorem}
  Let $u$ be a FB-solution to the stochastic Burgers equation. Then the assumptions of
  Proposition~\ref{prop:uniqueness} are satisfied and in particular the law
  of $u$ is unique.
\end{theorem}

\begin{proof}
  This is a combination of Lemmas~A.1--A.3 in~{\cite{GP16}}.
\end{proof}

Both the assumption of stationarity and the representation of the backward
process are only needed to control additive functionals of the FB-solution $u$
and thus to prove the convergence of $R^L$ and $Q^L$. If we had some other
means to prove this convergence, uniqueness would still follow. One
interesting situation where this is possible is the following: Let $u$ be an
energy solution of the stochastic Burgers equation which
satisfies~(\ref{eq:exp-control}), and assume that there exists a FB-solution
$v$ such that $\mathbb{P} \ll \mathbb{P}_{\tmop{FB}}$, where $\mathbb{P}$
denotes the law of $u$ on $C (\mathbb{R}_+, \CS')$ and
$\mathbb{P}_{\tmop{FB}}$ is the law of $v$. Then the convergence in
probability that we required in Proposition~\ref{prop:uniqueness} holds under
$\mathbb{P}_{\tmop{FB}}$ and thus also under $\mathbb{P}$ and therefore $u$ is still unique in law. Of
course, the assumption $\mathbb{P} \ll \mathbb{P}_{\tmop{FB}}$ is very
strong, but it can be verified in some nontrivial situations. For example in
{\cite{GJS15}} some particle systems are studied which start in an initial
condition $\mu^n$ that has bounded relative entropy $H (\mu^n| \nu)$ with
respect to a stationary distribution $\nu$, uniformly in $n$. Denoting the
distribution of the particle system started in $\pi$ by $\mathbb{P}_{\pi}$,
we get from the Markov property $H (\mathbb{P}_{\mu^n} |\mathbb{P}_{\nu}) =
H (\mu^n | \nu)$. Assume that the rescaled process converges to the law $\mathbb{P}_{\tmop{FB}}$ of a FB-solution under $\mathbb{P}_{\nu}$, and to the law of an energy solution $\P$ under $\P_{\mu^n}$. Then
\[ H (\mathbb{P}|\mathbb{P}_{\tmop{FB}}) \leqslant \liminf_{n \rightarrow
   \infty} H (\mu^n |{\nu}) < \infty, \]
and in particular $\mathbb{P} \ll \mathbb{P}_{\tmop{FB}}$; here we used that
the relative entropy is lower semicontinuous. Therefore, the scaling limits
of~{\cite{GJS15}} are still unique in law, even though they are not
FB-solutions. Let us summarize this observation:

\begin{theorem}
   Let $u$ be an energy solution to the stochastic Burgers equation which satisfies~\eqref{eq:exp-control} and let $v$ be an FB-solution. Denote the law of $u$ and $v$ with $\P$ and $\P_{\mathrm{FB}}$, respectively. Then the measure $\P$ is unique. In particular this applies for the non-stationary energy solutions constructed in~\cite{GJS15}.
\end{theorem}

\begin{remark}
   For simplicity we restricted our attention to the equation on $\T$, but everything above extends to the stochastic Burgers equation on $\R$. Also, once we understand the uniqueness of Burgers equation it is not difficult to also prove the uniqueness of its integral, the KPZ equation. For details see~\cite{GP16}.
\end{remark}

\end{document}